\documentclass[reqno]{amsart}

\newtheorem{theorem}{Theorem}[section]
\newtheorem{lemma}[theorem]{Lemma}
\newtheorem{corollary}[theorem]{Corollary}
\theoremstyle{definition}

\theoremstyle{remark}
\newtheorem{remark}[theorem]{Remark}

\numberwithin{equation}{section}

\newcommand{\di}{\displaystyle}

\begin{document}

\title[Euler Sums and a Theorem of Zagier]{Alternating Double Euler Sums, Hypergeometric Identities and a Theorem of Zagier}
\author{Lee-Peng Teo}

\address{Department of  Mathematics, Xiamen University Malaysia, Jalan Sunsuria, Bandar Sunsuria, 43900, Sepang, Selangor, Malaysia. }

\email{lpteo@xmu.edu.my}

\thanks{This work was started when I visited   J. H. Teh in National Tsing Hua University of Taiwan at the end of 2014. I would like to thank NTHU for their hospitality and J. H. Teh for the helpful discussions. }


\date{\today}


\begin{abstract}
In this work, we derive relations between generating functions of double stuffle relations and double shuffle relations to express the alternating double Euler sums $\zeta\left(\overline{r}, s\right)$, $\zeta\left(r, \overline{s}\right)$ and $\zeta\left(\overline{r}, \overline{s}\right)$ with $r+s$ odd in terms of zeta values. We also give a direct proof of a hypergeometric identity which is a limiting case of a basic hypergeometric identity of Andrews. Finally, we gave  another   proof for the formula of Zagier on the   multiple zeta values $\zeta(2,\ldots,2,3,2,\ldots,2)$.
\end{abstract}
\maketitle

\section{Introduction}

Following \cite{1}, for positive integers $k_1, \ldots, k_n$ with $k_n\geq 2$, define the multiple zeta values and multiple zeta star values by
\begin{align}
\zeta\left(k_1, \ldots, k_n\right)=&\sum_{1\leq m_1< \ldots< m_n }\frac{1}{m_1^{k_1}\ldots m_n^{k_n}},\label{eq1}\\
\zeta^{\star}\left(k_1, \ldots, k_n\right)=&\sum_{1\leq m_1\leq \ldots\leq m_n }\frac{1}{m_1^{k_1}\ldots m_n^{k_n}}.\label{eq2}
\end{align}
When $n=2$, the double sum
\begin{align}\label{eq3}
\zeta(r,s)=\sum_{m=2}^{\infty}\frac{1}{m^s}\sum_{j=1}^{k-1}\frac{1}{j^r }
\end{align}has been considered by Euler. Hence, multiple zeta values are also known as Euler sums.

As in \cite{6}, we can also define the following alternating double Euler sums
\begin{align}
\zeta(\overline{r},s)=&\sum_{m=2}^{\infty}\frac{1}{m^s}\sum_{j=1}^{m-1}\frac{(-1)^j}{j^r},\label{eq4}\\
\zeta(r, \overline{s})=&\sum_{m=2}^{\infty}\frac{(-1)^m}{m^s}\sum_{j=1}^{m-1}\frac{1}{j^r},\label{eq5}\\
\zeta(\overline{r}, \overline{s})=&\sum_{m=2}^{\infty}\frac{(-1)^m}{m^s}\sum_{j=1}^{m-1}\frac{(-1)^j}{j^r}.\label{eq6}
\end{align}
These can also be considered as extensions of the alternating series
\begin{align*}
\zeta(\overline{k})=&\sum_{m=1}^{\infty}\frac{(-1)^m}{m^k}\end{align*}to double Euler sums. Similar extensions can also be defined for multiple zeta values and multiple zeta star values.
It is well-known that
\begin{align}\label{eq7}\zeta(\overline{k})=-(1-2^{1-k})\zeta(k).\end{align}

In \cite{1}, Zagier studied the multiple zeta value $$H(a, b)=\zeta(\underbrace{2,\ldots,2}_a,3, \underbrace{2,\ldots,2}_b)$$ and the multiple zeta star value
 $$H^{\star}(a, b)=\zeta^{\star}(\underbrace{2,\ldots,2}_a,3, \underbrace{2,\ldots,2}_b).$$In particular, he proved the following formulas.

\begin{theorem}\label{Zagier}
Let
\begin{align*}
H(a)=&\zeta(\underbrace{2,\ldots,2}_a),\\
H^{\star}(a)=&\zeta^{\star}(\underbrace{2,\ldots,2}_a).
 \end{align*}
 Then
\begin{align}
H(a, b)=&2\sum_{r=1}^{K}(-1)^r\left\{\begin{pmatrix} 2r\\2a+2\end{pmatrix}\zeta(2r+1)+\begin{pmatrix} 2r\\2b+1\end{pmatrix}\zeta(\overline{2r+1})\right\}H(K-r),\label{eq8}\\
H^{\star}(a, b)=&-2\sum_{r=1}^{K} \left\{\left[\begin{pmatrix} 2r\\2a\end{pmatrix}-\delta_{r,a}\right]\zeta(2r+1)+\begin{pmatrix} 2r\\2b+1\end{pmatrix}\zeta(\overline{2r+1})\right\}H^{\star}(K-r),\label{eq9}
\end{align}where $K=a+b+1$.
\end{theorem}
The formula for $H(a, b)$ is needed in the proof of Hoffman conjecture by F. Brown \cite{13}, which states that all multiple zeta values can be expressed as $\mathbb{Q}$- linear combinations of the multiple zeta values $\zeta\left(k_1, \ldots, k_n\right)$ with each $k_i$ equals to 2 or 3.

The proof provided by Zagier for Theorem \ref{Zagier} uses    complex analytic methods to show that the generating functions of both sides of the formula are equal. In \cite{11}, Li provided an alternative proof using transformations of  hypergeometric series $_3F_2$. Recently, Pilehrood and Pilehrood \cite{3} gave another proof  using a special hypergeometric identity stated in \cite{2}, which can be considered as a limiting case of a basic hypergeometric identity of Andrews \cite{12}.   Pilehrood and Pilehrood \cite{3} proved the following result:
\begin{theorem}\label{Pilehrood}
\begin{align}\label{eq10}H^{\star}(a, b)=-4\zeta\left(2a+1, \overline{2b+2}\right)-2\zeta\left(\overline{2a+2b+3}\right).\end{align}
\end{theorem}
 As a matter of fact, this identity has been proved by Pilehrood, Pilehrood and Tauraso in \cite{16}. From this identity, Pilehrood and Pilehrood \cite{3}  proved  Theorem \ref{Zagier} using complex analytic method.

 As was pointed out in \cite{16}, the formula for $H^{\star}(a, b)$ given in Theorem \ref{Zagier} is actually an immediate consequence of Theorem \ref{Pilehrood} and one of the formulas in the following theorem.
\begin{theorem}
\label{doubleEuler} If $k=r+s\geq 3$ is an odd positive integer, then
\begin{equation}\label{eq11}\begin{split}
\zeta(r, s)
=&-\frac{1}{2}\zeta(k)
+\frac{1+(-1)^s}{2}\zeta(r)\zeta(s)\\&
+(-1)^r\sum_{l=0}^{\frac{k-1}{2}}\left[\begin{pmatrix}k-2l-1\\r-1\end{pmatrix}\zeta(k-2l)
+  \begin{pmatrix}k-2l-1\\s-1\end{pmatrix}\zeta(k-2l)\right]\zeta(2l),
\end{split}\end{equation}
\begin{equation}\label{eq12}\begin{split}
\zeta(\overline{r}, s)
=&-\frac{1}{2}\zeta(\overline{k})
+\frac{1+(-1)^s}{2}\zeta(\overline{r})\zeta(s)
\\&+(-1)^r\sum_{l=0}^{\frac{k-1}{2}}\left[\begin{pmatrix}k-2l-1\\r-1\end{pmatrix}\zeta(\overline{k-2l})
+ \begin{pmatrix}k-2l-1\\s-1\end{pmatrix}\zeta(k-2l)\right]\zeta(\overline{2l}),
\end{split}\end{equation}
\begin{equation}\label{eq13}\begin{split}
\zeta(r, \overline{s})=&-\frac{1}{2} \zeta(\overline{k})+\frac{1+(-1)^s}{2}\zeta(r)\zeta(\overline{s})
\\&+(-1)^r\sum_{l=0}^{\frac{k-1}{2}}\left[ \begin{pmatrix}k-2l-1\\r-1\end{pmatrix}\zeta(k-2l) + \begin{pmatrix}k-2l-1\\s-1\end{pmatrix}\zeta(\overline{k-2l})\right]\zeta(\overline{2l}),
\end{split}\end{equation}
\begin{equation}\label{eq14}\begin{split}
\zeta(\overline{r}, \overline{s})=&-\frac{1}{2}\zeta(k)
+\frac{1+(-1)^s}{2}\zeta(\overline{r})\zeta(\overline{s})
\\&+(-1)^r\sum_{l=0}^{\frac{k-1}{2}}\left[\begin{pmatrix}k-2l-1\\r-1\end{pmatrix}\zeta(\overline{k-2l})
+ \begin{pmatrix}k-2l-1\\s-1\end{pmatrix}\zeta(\overline{k-2l})\right]\zeta(2l).
\end{split}\end{equation}
\end{theorem}
One observes some symmetries  among these  four formulas. The first formula \eqref{eq11} in this theorem was proved in \cite{6} using matrices, and the other three  formulas \eqref{eq12}--\eqref{eq14} were mentioned in the same paper but the details of proofs   were not given. Proofs of these formulas were given in \cite{15} using contour integral representations. The formula for $H^{\star}(a, b)$ given in Theorem \ref{Zagier} is an immediate consequence of Theorem \ref{Pilehrood} and the third formula in Theorem \ref{doubleEuler}.

In this work, we are going to prove the   formulas \eqref{eq12}--\eqref{eq14} in Theorem \ref{doubleEuler} using the method used by Zagier \cite{1}  to prove the   formula \eqref{eq11}. This method has been outlined in \cite{8} and explored in \cite{10}. As pointed out in \cite{9}, the formulas in Theorem \ref{doubleEuler} can also be obtained by taking the $q\rightarrow 1$ limits of the corresponding $q$-analogs proved in \cite{14}. Nevertheless, we find it worthwhile to present the proof along the line of Zagier \cite{1}.

After proving the   formulas \eqref{eq12}--\eqref{eq14}, we give a direct proof to the hypergeometric identity used in \cite{3} to prove Theorem \ref{Pilehrood}, which is of interest in its own right. We then give a slightly simpler proof to Theorem \ref{Pilehrood}. From this, the formula \eqref{eq9} for $H^{\star}(a,b)$ in Theorem \ref{Zagier} follows immediately. We then prove the formula \eqref{eq8} for $H(a,b)$ in Theorem \ref{Zagier} from \eqref{eq9}, which better reflects the symmetries between these two formulas.

\section{Alternating Double Euler Sums}
In this section, we prove the formulas \eqref{eq12}--\eqref{eq14} in Theorem \ref{doubleEuler} using the method of Zagier \cite{1, 8}.
First we have the following well-known double-stuffle relations.
\begin{lemma}\label{lemma1}For $r\geq 1$, $s\geq 2$,
\begin{align}
\zeta(\overline{r})\zeta(s)=&\zeta(\overline{r},s)+\zeta(s, \overline{r})+\zeta(\overline{r+s}).\label{eq15}\end{align}For $r\geq 1$, $s\geq 1$,
\begin{align}
\zeta(\overline{r})\zeta(\overline{s})=&\zeta(\overline{r},\overline{s})+\zeta(\overline{s}, \overline{r})+\zeta(r+s).\label{eq16}
\end{align}
\end{lemma}
\begin{proof}We note that $\zeta(\overline{1})=\log 2$ is well-defined.
Eq. \eqref{eq15} follows from
\begin{align*}
\sum_{m=1}^{\infty}\frac{(-1)^m}{m^r}\sum_{j=1}^{\infty}\frac{1}{j^s}=\sum_{m=1}^{\infty}\frac{(-1)^m}{m^r}\sum_{j=1}^{m-1}\frac{1}{j^s}+
\sum_{m=1}^{\infty}\frac{(-1)^m}{m^r}\frac{1}{m^s}+\sum_{m=1}^{\infty}\frac{(-1)^m}{m^r}\sum_{j=m+1}^{\infty}\frac{1}{j^s}.
\end{align*} Eq. \eqref{eq16} is proved in the same way.
\end{proof}

When $r\geq 1$, $\zeta(r, \overline{1})$ and $\zeta(\overline{r}, \overline{1})$ are convergent. However, $\zeta(\overline{r}, s)$ is convergent only if $s\geq 2$. As in \cite{1}, we define $\zeta(1)$ symbolically as $T$. For $r\geq 1$, let
\begin{align*}
\zeta(\overline{r}, 1)=\zeta(\overline{r})T-\zeta(1,\overline{r})-\zeta(\overline{r+1}).
\end{align*}Then the relation \eqref{eq15} still holds symbolically for $s=1$.

Fix an integer $k$,  we define the following generating functions:
\begin{align*}
F_1(x,y)=&\sum_{\substack{r, s\geq 1\\ r+s=k}} \zeta(\overline{r})\zeta(s)x^{r-1}y^{s-1},\\
F_2(x,y)=&\sum_{\substack{r, s\geq 1\\ r+s=k}}  \zeta(\overline{r})\zeta(\overline{s})x^{r-1}y^{s-1},\end{align*}\begin{align*}
G_1(x, y)=&\sum_{\substack{r, s\geq 1\\ r+s=k}}  \zeta(\overline{r}, s)x^{r-1}y^{s-1},\\
G_2(x, y)=&\sum_{\substack{r, s\geq 1\\ r+s=k}}  \zeta(r, \overline{s})x^{r-1}y^{s-1},\\G_3(x, y)=&\sum_{\substack{r, s\geq 1\\ r+s=k}}  \zeta(\overline{r}, \overline{s})x^{r-1}y^{s-1},\end{align*}\begin{align*}
T_1(x,y)=&\zeta(k)\sum_{\substack{r, s\geq 1\\ r+s=k}} x^{r-1}y^{s-1}=\zeta(k)\frac{x^{k-1}-y^{k-1}}{x-y},\\
T_2(x,y)=&\zeta(\overline{k})\sum_{\substack{r, s\geq 1\\ r+s=k}} x^{r-1}y^{s-1}=\zeta(\overline{k})\frac{x^{k-1}-y^{k-1}}{x-y}.
\end{align*}
Notice that only $F_2(x,y)$, $T_1(x,y)$ and $T_2(x,y)$  are symmetric in $x$ and $y$.
The relations in Lemma \ref{lemma1} translate into the following identities:
\begin{theorem}\label{R1}
\begin{align}
F_1(x,y)=G_1(x,y)+G_2(y, x)+T_2(x,y),\label{eq17}\\
F_2(x,y)=G_3(x,y)+G_3(y, x)+T_1(x,y).\label{eq18}
\end{align}
\end{theorem}

Next, we derive the double-shuffle relation for alternating double Euler sums.
Notice that
\begin{align*}
\zeta(r)=&\sum_{m=1}^{\infty}\frac{1}{m^r}\\
=&\frac{1}{\Gamma(r)}\int_0^{\infty} t^{r-1}\sum_{n=1}^{\infty} e^{-tn}dt\\
=&\frac{1}{\Gamma(r)}\int_0^{\infty}\frac{t^{r-1}}{e^t-1}dt.
\end{align*}
In the same way, one can derive the formula
\begin{align*}
\zeta(\overline{r})=& -\frac{1}{\Gamma(r)}\int_0^{\infty}\frac{t^{r-1}}{e^t+1}dt.
\end{align*}For alternating double zeta values, we have
\begin{lemma}\label{lemma5}
\begin{align}
\zeta(\overline{r}, s)=&-\frac{1}{\Gamma(r)\Gamma(s)}\int_0^{\infty}\int_0^{\infty}\frac{t^{s-1}u^{r-1}}{(e^t-1)(e^{t+u}+1) }dudt,\label{eq19}\\
\zeta\left(r, \overline{s} \right)=&\frac{1}{\Gamma(r)\Gamma(s)}\int_0^{\infty}\int_0^{\infty}\frac{t^{s-1}u^{r-1}}{(e^t+1)(e^{t+u}+1) }dudt,\label{eq20}\\
\zeta\left(\overline{r}, \overline{s} \right)=&-\frac{1}{\Gamma(r)\Gamma(s)}\int_0^{\infty}\int_0^{\infty}\frac{t^{s-1}u^{r-1}}{(e^t+1)(e^{t+u}-1) }dudt.\label{eq21}
\end{align}
\end{lemma}
\begin{proof}
\begin{align*}
\zeta\left(\overline{r}, s\right)=&\sum_{m=1}^{\infty} \frac{1}{m^s}\sum_{j=1}^{m-1}\frac{(-1)^j}{j^r}\\
=&\sum_{j=1}^{\infty}\frac{(-1)^j}{j^r}\sum_{m=1}^{\infty} \frac{1}{(m+j)^s}\\
=&\frac{1}{\Gamma(r)\Gamma(s)}\int_0^{\infty}\int_0^{\infty}t^{s-1}u^{r-1} \sum_{j=1}^{\infty}(-1)^j\sum_{m=1}^{\infty} e^{-t(m+j)}e^{-uj}dudt\\
=&-\frac{1}{\Gamma(r)\Gamma(s)}\int_0^{\infty}\int_0^{\infty}\frac{t^{s-1}u^{r-1}}{(e^t-1)(e^{t+u}+1) }dudt.
\end{align*}
The other two formulas are proved in the same way.
\end{proof}

For the usual double Euler sum, the shuffle relation reads as
\begin{align}
\zeta(r)\zeta(s)=&\sum_{j=1}^{k-1} \begin{pmatrix}j-1\\r-1\end{pmatrix} \zeta(k-j, j)+\sum_{j=1}^{k-1} \begin{pmatrix}j-1\\s-1\end{pmatrix} \zeta(k-j, j)
\end{align}when $r\geq 1$, $s\geq 2$ and $r+s=k$.
Using the integral representations given in Lemma \ref{lemma5}, we can prove the following shuffle relations for alternating double Euler sums:
\begin{theorem}\label{shuffle}
If $r\geq 1$, $s\geq 2$, $r+s=k$, then
\begin{align}
\zeta(\overline{r})\zeta(s)=&\sum_{j=1}^{k-1} \begin{pmatrix}j-1\\r-1\end{pmatrix} \zeta(\overline{k-j}, \overline{j})+\sum_{j=1}^{k-1} \begin{pmatrix}j-1\\s-1\end{pmatrix} \zeta(\overline{k-j}, j).\label{eq22}
\end{align}
If $r, s\geq 1$, $r+s=k$, then
\begin{align}\label{eq23}
\zeta(\overline{r})\zeta(\overline{s})=&\sum_{j=1}^{k-1}  \begin{pmatrix}j-1\\r-1\end{pmatrix} \zeta(k-j, \overline{j})+\sum_{j=1}^{k-1}  \begin{pmatrix}j-1\\s-1\end{pmatrix} \zeta(k-j, \overline{j}).
\end{align}
\end{theorem}
\begin{proof}
\begin{align*}
&\zeta(\overline{r})\zeta(s)\\=&-\frac{1}{\Gamma(r) \Gamma(s)}\int_0^{\infty}\frac{t^{r-1}}{e^t+1}dt\int_0^{\infty}\frac{u^{s-1}}{e^u-1}du\\
=&-\frac{1}{\Gamma(r) \Gamma(s)}\int_0^{\infty}\int_t^{\infty}\frac{t^{r-1}u^{s-1}}{(e^t+1)(e^u-1)}dudt-\frac{1}{\Gamma(r) \Gamma(s)}\int_0^{\infty}\int_u^{\infty}\frac{t^{r-1}u^{s-1}}{(e^t+1)(e^u-1)}dtdu\\
=&-\frac{1}{\Gamma(r)\Gamma(s)}\int_0^{\infty}\int_0^{\infty}\frac{t^{r-1}(t+u)^{s-1}}{(e^t+1)(e^{t+u}-1)}dsdt-\frac{1}{\Gamma(r)\Gamma(s)}
\int_0^{\infty}\int_0^{\infty}\frac{(t+u)^{r-1}u^{s-1}}{(e^{t+u}+1)(e^u-1)}dtdu\\
=&\sum_{j=0}^{s-1}\begin{pmatrix}s-1\\j\end{pmatrix}\frac{\Gamma(s-j)\Gamma(r+j)}{\Gamma(r) \Gamma(s)}\zeta(\overline{s-j},\overline{r+j})+ \sum_{j=0}^{r-1}\begin{pmatrix}r-1\\j\end{pmatrix}\frac{\Gamma(r-j)\Gamma(s+j)}{\Gamma(r) \Gamma(s)}\zeta(\overline{r-j}, s+j)\\
=&\sum_{j=0}^{s-1}\begin{pmatrix}r+j-1\\r-1\end{pmatrix} \zeta(\overline{s-j},\overline{r+j})+ \sum_{j=0}^{r-1}\begin{pmatrix}s+j-1\\s-1\end{pmatrix} \zeta(\overline{r-j}, s+j)\\
=&\sum_{j=1}^{k-1} \begin{pmatrix}j-1\\r-1\end{pmatrix} \zeta(\overline{k-j}, \overline{j})+\sum_{j=1}^{k-1} \begin{pmatrix}j-1\\s-1\end{pmatrix} \zeta(\overline{k-j}, j).
\end{align*}
The  formula \eqref{eq23} is proved in the same way.
\end{proof}
Before translating these two identities into identities for generating functions, we need to deal with the case $\zeta(\overline{r})\zeta(s)$ with $s=1$.
\begin{lemma}If $k\geq 3$,
\begin{align}
\sum_{s=2}^{k-1}\zeta(k-s, s)=&\zeta(k),\label{eq24}\\
\sum_{s=2}^{k-1}\zeta(\overline{k-s}, s)=&\zeta(\overline{k})+\zeta(1, \overline{k-1})-\zeta(\overline{1}, \overline{k-1}).\label{eq25}
\end{align}
\end{lemma}
\begin{proof}
By definition,
\begin{align*}
\sum_{s=2}^{k-1}\zeta(k-s, s)=&\sum_{1\leq m<n}\sum_{s=2}^{k-1}\frac{1}{m^{k-s}}\frac{1}{n^s}\\
=&\sum_{1\leq m<n} \frac{1}{m^{k}}\frac{\di \frac{m^2}{n^2}\left(1-\left(\frac{m}{n}\right)^{k-2}\right)}{\di 1-\frac{m}{n}}\\
=&\sum_{n=2}^{\infty}\sum_{m=1}^{n-1} \frac{1}{m^{k-2}n(n-m)}-\sum_{n=2}^{\infty}\sum_{m=1}^{n-1}\frac{1}{n^{k-1}(n-m)}.
\end{align*}Notice that
\begin{align*}
\sum_{n=2}^{\infty}\sum_{m=1}^{n-1}\frac{1}{n^{k-1}(n-m)}=&\sum_{n=2}^{\infty}\sum_{m=1}^{n-1}\frac{1}{n^{k-1}m}=\zeta(1, k-1).
\end{align*}On the other hand,
\begin{align*}
&\sum_{n=2}^{\infty}\sum_{m=1}^{n-1}\frac{1}{m^{k-2}n(n-m)}\\=&\sum_{m=1}^{\infty}\sum_{n=1}^{\infty} \frac{1}{m^{k-2}n(n+m)}\\
=&\lim_{L\rightarrow\infty}\sum_{m=1}^{\infty}\sum_{n=1}^L\frac{1}{m^{k-1}}\left(\frac{1}{n}-\frac{1}{n+m}\right)\\
=&\lim_{L\rightarrow\infty}\sum_{m=1}^{\infty} \frac{1}{m^{k-1}}\left(1+\frac{1}{2}+\ldots+\frac{1}{m}-\frac{1}{L+1}-\ldots-\frac{1}{L+m}\right)\\
=&\sum_{m=1}^{\infty} \sum_{n=1}^{m-1}\frac{1}{m^{k-1}}\frac{1}{n}+\sum_{m=1}^{\infty} \frac{1}{m^{k}}-\lim_{L\rightarrow\infty}\sum_{m=1}^{\infty} \frac{1}{m^{k-1}}\left( \frac{1}{L+1}+\ldots+\frac{1}{L+m}\right).
\end{align*}
The first two terms give
\begin{align*}
\sum_{m=1}^{\infty} \sum_{n=1}^{m-1}\frac{1}{m^{k-1}}\frac{1}{n}+\sum_{m=1}^{\infty} \frac{1}{m^{k}}=&\zeta(1, k-1)+\zeta(k).
\end{align*}
For the term that involves $L$, notice that
\begin{align*}
0\leq & \sum_{m=1}^{\infty} \frac{1}{m^{k-1}}\left( \frac{1}{L+1}+\ldots+\frac{1}{L+m}\right) \\
=&  \sum_{m=1}^{\infty} \frac{1}{m^{k-1}}\sum_{s=1}^m\frac{1}{L+s} \\
=& \sum_{s=1}^{\infty}\frac{1}{L+s}\sum_{m=s}^{\infty}\frac{1}{m^{k-1}} \\
 \leq  & \sum_{s=1}^{\infty}\frac{1}{L+s}\left(\frac{1}{s^{k-1}}+\int_{s}^{\infty}\frac{1}{x^{k-1}}dx\right)\\
 =&\sum_{s=1}^{\infty}\frac{1}{L+s} \frac{1}{s^{k-1}}+\frac{1}{k-2}\sum_{s=1}^{\infty}\frac{1}{L+s} \frac{1}{s^{k-2}}\\
\leq & \frac{1}{L}\sum_{s=1}^{\infty} \frac{1}{s^{k-1}}+\frac{1}{k-2}\sum_{s=1}^{\infty}\frac{1}{2\sqrt{Ls}} \frac{1}{s^{k-2}}.
\end{align*}
Hence, this term goes to 0 uniformly as $L\rightarrow\infty$.
 This proves that
 $$\sum_{s=2}^{k-1}\zeta(k-s, s)=\zeta(k).$$

Equation \eqref{eq25} is proved in a similarly way. We have
\begin{align*}
\sum_{s=2}^{k-1}\zeta(\overline{k-s}, s)
=&\sum_{n=2}^{\infty}\sum_{m=1}^{n-1} \frac{(-1)^m}{m^{k-2}n(n-m)}-\sum_{n=2}^{\infty}\sum_{m=1}^{n-1}  \frac{(-1)^m}{n^{k-1}(n-m)},
\end{align*}
\begin{align*}
\sum_{n=2}^{\infty}\sum_{m=1}^{n-1}  \frac{(-1)^m}{n^{k-1}(n-m)}=&\sum_{n=2}^{\infty}\sum_{m=1}^{n-1}  \frac{(-1)^{n-m}}{n^{k-1}m}=\zeta(\overline{1}, \overline{k-1}),
\end{align*}
\begin{align*}
\sum_{n=2}^{\infty}\sum_{m=1}^{n-1}  \frac{(-1)^m}{m^{k-2}n(n-m)}
=&\sum_{m=2}^{\infty} \sum_{n=1}^{m-1}\frac{(-1)^m}{m^{k-1}}\frac{1}{n}+\sum_{m=1}^{\infty} \frac{(-1)^m}{m^{k}}\\&-\lim_{L\rightarrow\infty}\sum_{m=1}^{\infty} \frac{(-1)^m}{m^{k-1}}\left( \frac{1}{L+1}+\ldots+\frac{1}{L+m}\right).
\end{align*}The first two terms give
\begin{align*}
\sum_{m=2}^{\infty} \sum_{n=1}^{m-1}\frac{(-1)^m}{m^{k-1}}\frac{1}{n}+\sum_{m=1}^{\infty} \frac{(-1)^m}{m^{k}}=\zeta(1,\overline{k-1})+\zeta(\overline{k}).
\end{align*}
For the term that involves $L$, since
\begin{align*}
&\left| \sum_{m=1}^{\infty} \frac{(-1)^m}{m^{k-1}}\left( \frac{1}{L+1}+\ldots+\frac{1}{L+m}\right)\right|
\leq  \sum_{m=1}^{\infty} \frac{1}{m^{k-1}}\left( \frac{1}{L+1}+\ldots+\frac{1}{L+m}\right).
\end{align*}It follows from the previous estimate that this term goes to 0 as $L\rightarrow \infty$.
This proves \eqref{eq25}.
\end{proof}
\vfill
\pagebreak
\begin{corollary}
The relation \eqref{eq22} holds formally when $s=1$.
\end{corollary}
\begin{proof}
Put $s=1$ and $r=k-1$ into \eqref{eq22} gives
\begin{align}\label{eq26}
\zeta(\overline{k-1})\zeta(1)=& \zeta(\overline{1}, \overline{k-1})+\sum_{j=1}^{k-1}  \zeta(\overline{k-j}, j).
\end{align}
Since formally
$$\zeta(\overline{k-1})\zeta(1)=\zeta(\overline{k-1}, 1)+\zeta(1, \overline{k-1})+\zeta(\overline{k})$$
Eq. \eqref{eq26} is equivalent to
\begin{align*}
\sum_{j=2}^{k-1}  \zeta(\overline{k-j}, j)=\zeta(1, \overline{k-1})+\zeta(\overline{k})-\zeta(\overline{1}, \overline{k-1}),
\end{align*}which is equation \eqref{eq25}. This proves the corollary.
\end{proof}

\begin{remark}
Setting $r=1$ in \eqref{eq22}, we have
\begin{align*}
\zeta(k-1)\zeta(\overline{1})=&\sum_{j=1}^{k-1}   \zeta(\overline{k-j}, \overline{j})+\zeta(\overline{1}, k-1).
\end{align*}
Together with the fact that
$$\zeta(k-1)\zeta(\overline{1})-\zeta(\overline{1}, k-1)=\zeta(\overline{k})+\zeta(k-1, \overline{1}),$$ we obtain
\begin{align}\label{eq27}
\sum_{s=2}^{k-1}   \zeta(\overline{k-s}, \overline{s})=&\zeta(\overline{k})+\zeta(k-1,\overline{1})-\zeta(\overline{k-1}, \overline{1}).
\end{align}
Similarly, setting $r=1$ in \eqref{eq23} gives
\begin{align}\label{eq28}
\sum_{s=2}^{k-1}   \zeta(k-s, \overline{s})=&\zeta(\overline{k})+ \zeta(\overline{k-1}, \overline{1})+\zeta(\overline{1}, \overline{k-1})-\zeta(k-1, \overline{1})-\zeta(1, \overline{k-1}).
\end{align}
Eq. \eqref{eq25}, \eqref{eq27} and \eqref{eq28} are summation formulas of alternating double Euler sums. They are generalizations of \eqref{eq24}.
\end{remark}

The double shuffle relations in Theorem \ref{shuffle} give the following identities of generating functions.
\begin{theorem}\label{R2}
\begin{align}
F_1(x,y)=&G_1(x,x+y)+G_3(y, x+y),\label{eq29}\\
F_2(x,y)=&G_2(x,x+y)+G_2(y, x+y),\label{eq30}.
\end{align}
\end{theorem}
\begin{proof}
Eq. \eqref{eq22} gives
\begin{align*}
F_1(x,y)=&\sum_{\substack{r,s\geq 1\\r+s=k}}\zeta(\overline{r})\zeta(s)x^{r-1}y^{s-1}\\
=&\sum_{r=1}^{k-1}\sum_{j=r}^{k-1} \begin{pmatrix}j-1\\r-1\end{pmatrix} \zeta(\overline{k-j}, \overline{j})x^{r-1}y^{k-r-1}+\sum_{s=1}^{k-1}\sum_{j=s}^{k-1} \begin{pmatrix}j-1\\s-1\end{pmatrix} \zeta(\overline{k-j}, j)x^{k-s-1}y^{s-1}\\
=&\sum_{j=1}^{k-1}\left[\sum_{r=1}^j\begin{pmatrix}j-1\\r-1\end{pmatrix}x^{r-1}y^{j-r}\right]\zeta(\overline{k-j}, \overline{j})y^{k-j-1}
\\&+\sum_{j=1}^{k-1}\left[\sum_{s=1}^j\begin{pmatrix}j-1\\s-1\end{pmatrix}x^{j-s}y^{s-1}\right]\zeta(\overline{k-j}, j)x^{k-j-1}\\
=&\sum_{j=1}^{k-1}\zeta(\overline{k-j}, \overline{j})y^{k-j-1}(x+y)^{j-1}+\sum_{j=1}^{k-1}\zeta(\overline{k-j}, j)x^{k-j-1}(x+y)^{j-1}\\
=&G_3(y, x+y)+G_1(x, x+y)
\end{align*}
Eq. \eqref{eq30} is proved in the same way.
\end{proof}
From Theorem \ref{R1} and Theorem \ref{R2}, one obtains four relations among the functions $F_1(x, y)$, $F_2(x, y)$, $G_1(x, y)$, $G_2(x, y)$ and $G_3(x, y)$. Our goal is to express $G_1$, $G_2$ and $G_3$ in terms of $F_1$ and $F_2$.
\begin{theorem}
\label{R3}We have the following relations:
\begin{equation}\label{eq31}\begin{split}
G_1(x,y)-G_1(-x,-y)=&F_1(x,y)-F_1(x,-y)-F_2(x-y,y)+F_2(x-y, -y)\\&+F_1(x,x-y)-F_1(-x,x-y)\\&-T_2(x,y)-T_2(x,x-y)-T_1(x-y,-y),
\end{split}\end{equation}
\begin{equation}\label{eq32}\begin{split}
G_2(x,y)-G_2(-x,-y)=&F_1(y,x)-F_1(-y,x)-F_1(y,x-y)+F_1(-y,x-y)\\&+F_2( x, x-y)-F_2(-x,x-y)\\&-T_2(x,y)-T_1(x,x-y)-T_2(x-y,-y),
\end{split}\end{equation}
\begin{equation}\label{eq33}\begin{split}
G_3(x,y)-G_3(-x,-y)=&F_2(x,y)-F_2(x,-y)-F_1(x-y,y)+F_1(x-y, -y)\\
&+F_1(x-y, x)-F_1(x-y, -x)\\&-T_1(x,y)-T_2(x,x-y)-T_2(x-y,-y).
\end{split}\end{equation}
\end{theorem}

\begin{proof}As in \cite{1}, we use the relations in Theorem \ref{R1} and Theorem \ref{R2} alternately. We have
\begin{align*}
G_1(x,y)=&F_1(x,y)-G_2(y,x)-T_2(x,y)\\
=&F_1(x,y)-F_2(x-y, y)+G_2(x-y,x)-T_2(x,y)\\
=&F_1(x,y)-F_2( x-y, y)+F_1(x, x-y)-G_1(x,x-y)-T_2(x,y)-T_2(x, x-y)\end{align*}\begin{align*}
=&F_1(x,y)-F_2( x-y, y)+F_1(x, x-y)-F_1(x,-y)+G_3(-y,x-y)\\
&-T_2(x,y)-T_2(x, x-y)\\
=&F_1(x,y)-F_2( x-y, y)+F_1(x, x-y)-F_1(x,-y)+F_2(-y,x-y)\\
&-G_3(x-y, -y)-T_2(x,y)-T_2(x, x-y)-T_1(-y, x-y)\\
=&F_1(x,y)-F_2( x-y, y)+F_1(x, x-y)-F_1(x,-y)+F_2(-y,x-y)\\
&-F_1(-x,x-y)+G_1(-x,-y)-T_2(x,y)-T_2(x, x-y)-T_1(-y, x-y).
\end{align*}
This proves   \eqref{eq31}. The other two equations can be proved in the same way.
\end{proof}

From Theorem \ref{R3}, we can obtain the main results of this section.
\begin{theorem}
\label{doubleEuler2} If $k=r+s\geq 3$ is an odd positive integer, then
\begin{equation}\label{eq34}\begin{split}
\zeta(\overline{r}, s)
=&-\frac{1}{2}\zeta(\overline{k})
+\frac{1+(-1)^s}{2}\zeta(\overline{r})\zeta(s)
\\&+(-1)^r\sum_{l=0}^{\frac{k-1}{2}}\left[\begin{pmatrix}k-2l-1\\r-1\end{pmatrix}\zeta(\overline{k-2l})
+ \begin{pmatrix}k-2l-1\\s-1\end{pmatrix}\zeta(k-2l)\right]\zeta(\overline{2l}),
\end{split}\end{equation}
\begin{equation}\label{eq35}\begin{split}
\zeta(r, \overline{s})=&-\frac{1}{2} \zeta(\overline{k})+\frac{1+(-1)^s}{2}\zeta(r)\zeta(\overline{s})
\\&+(-1)^r\sum_{l=0}^{\frac{k-1}{2}} \left[\begin{pmatrix}k-2l-1\\r-1\end{pmatrix}\zeta(k-2l) + \begin{pmatrix}k-2l-1\\s-1\end{pmatrix}\zeta(\overline{k-2l})\right]\zeta(\overline{2l}),
\end{split}\end{equation}
\begin{equation}\label{eq36}\begin{split}
\zeta(\overline{r}, \overline{s})=&-\frac{1}{2}\zeta(k)
+\frac{1+(-1)^s}{2}\zeta(\overline{r})\zeta(\overline{s})
\\&+(-1)^r\sum_{l=0}^{\frac{k-1}{2}}\left[\begin{pmatrix}k-2l-1\\r-1\end{pmatrix}\zeta(\overline{k-2l})
+ \begin{pmatrix}k-2l-1\\s-1\end{pmatrix}\zeta(\overline{k-2l})\right]\zeta(2l).
\end{split}\end{equation}
\end{theorem}
\begin{proof}
We prove \eqref{eq34}. The other two formulas can be derived in the same way. We apply the formula \eqref{eq31}. Since $k$ is odd,
\begin{align*}
G_1(x,y)-G_1(-x,-y)=&2\sum_{\substack{r, s\geq 1\\r+s=k}}\zeta(\overline{r}, s)x^{r-1}y^{s-1}.
\end{align*}
\begin{align*}
F_1(x,y)-F_1(x, -y)=& \sum_{\substack{r, s\geq 1\\r+s=k}}\left(1+(-1)^s\right)\zeta(\overline{r})\zeta(s)x^{r-1}y^{s-1}.
\end{align*}
\begin{align*}
F_2(x-y,-y)-F_2(x-y,y)=&-2\sum_{l=1}^{\frac{k-1}{2}}\zeta(\overline{k-2l})\zeta(\overline{2l})(x-y)^{k-2l-1}  y^{2l-1}\\
=&-2\sum_{l=1}^{\frac{k-1}{2}}\zeta(\overline{k-2l})\zeta(\overline{2l})\sum_{r=1}^{k-1}(-1)^{k-2l-r}\begin{pmatrix}k-2l-1\\r-1\end{pmatrix} x^{r-1}  y^{k-r-1}\\
=&2\sum_{r=1}^{k-1}\left(\sum_{l=1}^{\frac{k-1}{2}}(-1)^{r}\begin{pmatrix}k-2l-1\\r-1\end{pmatrix}\zeta(\overline{k-2l})\zeta(\overline{2l})\right) x^{r-1}  y^{k-r-1}.
\end{align*}
Similarly,
\begin{align*}
F_1(x,x-y)-F_1(-x,x-y) =&2\sum_{r=1}^{k-1}\left(\sum_{l=1}^{\frac{k-1}{2}}(-1)^r\begin{pmatrix}k-2l-1\\k-r-1\end{pmatrix}\zeta(k-2l)\zeta(\overline{2l})\right) x^{r-1}  y^{k-r-1}.
\end{align*}
\begin{align*}
T_2(x,x-y)=&\zeta(\overline{k})\frac{x^{k-1}-(x-y)^{k-1}}{y} \\
=&\zeta(\overline{k})\sum_{r=1}^{k-1} (-1)^r\begin{pmatrix} k-1\\r-1\end{pmatrix} x^{r-1}y^{k-r-1},\\
T_1(x-y,-y)=&\zeta(k)\frac{(x-y)^{k-1}-(-y)^{k-1}}{x}\\
=&\zeta(k)\sum_{r=1}^{k-1}(-1)^r\begin{pmatrix} k-1\\k-r-1\end{pmatrix}x^{r-1}y^{k-r-1}.
\end{align*}
Compare both sides of \eqref{eq31} and use the fact that
$$\zeta(0)=\zeta(\overline{0})=-\frac{1}{2}$$ give
\begin{align*}
\zeta(\overline{r}, s)
=&-\frac{1}{2}\zeta(\overline{k})
+\frac{1+(-1)^s}{2}\zeta(\overline{r})\zeta(s)
\\&+(-1)^r\sum_{l=0}^{\frac{k-1}{2}}\begin{pmatrix}k-2l-1\\r-1\end{pmatrix}\zeta(\overline{k-2l})\zeta(\overline{2l})
+(-1)^r\sum_{l=0}^{\frac{k-1}{2}} \begin{pmatrix}k-2l-1\\s-1\end{pmatrix}\zeta(k-2l)\zeta(\overline{2l}),
\end{align*}which is the desired result.
\end{proof}

From the proof, we see that in deriving the formulas for the alternating Euler sums, we need to use relations that involve $\zeta(\overline{r}, s)$, $\zeta(r, \overline{s})$ and $\zeta(\overline{r}, \overline{s})$ together. Considering anyone of them alone cannot work.

\begin{remark}
In \eqref{eq34}, $r\geq 1$, $s\geq 2$ and both sides are well-defined. Similarly, in \eqref{eq36}, $r\geq 1$, $s\geq 1$ and both sides are also well-defined. In \eqref{eq35}, one has to be careful when $r=1$. In this case, the equation should read as
\begin{equation}\label{eq57}\begin{split}
\zeta(r, \overline{s})=&-\frac{1}{2} \zeta(\overline{k})
\\&+(-1)^r\left[\sum_{l=0}^{\frac{k-3}{2}} \begin{pmatrix}k-2l-1\\r-1\end{pmatrix}\zeta(k-2l)\zeta(\overline{2l}) + \sum_{l=0}^{\frac{k-1}{2}} \begin{pmatrix}k-2l-1\\s-1\end{pmatrix}\zeta(\overline{k-2l})\zeta(\overline{2l})\right].
\end{split}\end{equation}

\end{remark}

\section{Hypergeometric Identities }
In this section, we prove a hypergeometric identity that is of interest in its own right. Recall that the generalized hypergeometric function is defined as (see e.g. \cite{4}):
\begin{align*}
&_{p}F_{q}\left[\begin{aligned} a_1, a_2,   \ldots, a_p\\b_1, b_2, \ldots, b_q\end{aligned}\;;\;x\right]=\sum_{n=0}^{\infty}
\frac{(a_1)_n\ldots (a_p)_n}{(b_1)_n\ldots (b_q)_n}\frac{x^n}{n!}.
\end{align*}Here $$(x)_n=x(x+1)\ldots(x+n-1)=\frac{\Gamma(x+n)}{\Gamma(x)}$$
is the Pochhammer symbol, and $b_i$ are not negative integers or zero. If $p=q+1$, the series converges absolutely for $|x|<1$. In addition, if
$$\text{Re}\,\left(\sum_{i=1}^q b_i-\sum_{i=1}^{q+1}a_i\right)>0,$$
then the series $_{q+1}F_q$ converges absolutely for $|x|\leq 1$. If
$$0\geq \text{Re}\,\left(\sum_{i=1}^q b_i-\sum_{i=1}^{q+1}a_i\right)>-1,$$
the series $_{q+1}F_q$ converges conditionally if $|x|= 1$ and $x\neq 1$.

Of special interest are the hypergeometric sums
$\displaystyle
_{q+1}F_{q}\left[\begin{aligned} a_1, a_2,   \ldots, a_{q+1}\\b_1, b_2, \ldots, b_q\quad\end{aligned}\,;\,\pm 1\right]$. We first quote a few well-known identities.
\begin{lemma}[Gauss]\label{lemma7}
If $\text{Re}\,(c-a-b)>0$, then
\begin{align*}
\sum_{n=0}^{\infty}\frac{(a)_n(b)_n}{n! (c)_n}=\,_2F_1\left[\begin{aligned} a, b\\ c\;\;\end{aligned}\;;\;1\right]=\frac{\Gamma(c)\Gamma(c-a-b)}{\Gamma(c-a)\Gamma(c-b)}.
\end{align*}
\end{lemma}
\begin{proof}
See \cite{4}, page 66.
\end{proof}

\begin{lemma}[Pfaff-Saalsch\"utz]\label{lemma8}
\begin{align}
_3F_2\left[\begin{aligned} a, b, -n \hspace{1cm}\\c, 1+a+b-c-n\end{aligned}\;;\;1\right]=\frac{(c-a)_n(c-b)_n}{(c)_n(c-a-b)_n}.
\end{align}
\end{lemma}
\begin{proof}
See \cite{4}, page 69.
\end{proof}
As a corollary, we have
\begin{corollary}\label{corollary2}
\begin{align*}
\frac{(b)_n(c)_n}{(1+a-b)_n(1+a-c)_n}=&\sum_{r=0}^n\frac{(a+n)_r (1+a-b-c)_r(-n)_r}{r! (1+a-b)_r(1+a-c)_r}.
\end{align*}
\end{corollary}
\begin{proof}
Using the fact that
\begin{align*}
(x)_n=&x(x+1)\ldots(x+n-1)\\=&(-1)^n(1-x-n)(2-x-n)\ldots (-x)\\=&(-1)^n(1-x-n)_n,\end{align*}
we have
\begin{align*}
\frac{(b)_n(c)_n}{(1+a-b)_n(1+a-c)_n}=&\frac{(b)_n(1-c-n)_n}{(b-a-n)_n(1+a-c)_n}\\
=&_3F_2\left[\begin{aligned} a+n, 1+a-b-c, -n \\1+a-b, 1+a-c\quad \end{aligned}\;;\;1\right]\\
=&\sum_{r=0}^n\frac{(a+n)_r (1+a-b-c)_r(-n)_r}{r! (1+a-b)_r(1+a-c)_r}.
\end{align*}
\end{proof}

The following is an identity that can be proved using integral representation of $_2F_1$.
\begin{lemma}\label{lemma9} If $1+a-b$ is not zero or a negative integer, and $\di\text{Re}\, b<1$,
\begin{align*}_2F_1\left[\begin{aligned} a, b\quad \\1+a-b\end{aligned}\;;\;-1\right]=&\frac{1}{2}\frac{\di\Gamma\left(\frac{a}{2}\right)\Gamma(1+a-b)}{\Gamma(a)\di\Gamma\left(1+\frac{a}{2}-b\right)}
\end{align*}
\end{lemma}
\begin{proof}
Using the integral representation (see \cite{4}, page 65)
\begin{align*}
_2F_1\left[\begin{aligned} \alpha, \beta\\ \gamma\;\;\end{aligned}\;;\;x\right]=&\frac{\Gamma(\gamma)}{\Gamma(\beta)\Gamma(\gamma-\beta)}\int_0^1t^{\beta-1}(1-t)^{\gamma-\beta-1}(1-xt)^{-\alpha}dt,
\end{align*}we find that if $\text{Re}\,a>0$ and $\text{Re}\, b<1$,
\begin{align*}
_2F_1\left[\begin{aligned} a, b\quad \\1+a-b\end{aligned}\;;\;-1\right]=&\frac{\Gamma(1+a-b)}{\Gamma(a)\Gamma(1-b)}\int_0^1t^{a-1}(1-t)^{-b}(1+t)^{-b}dt\\
=&\frac{1}{2}\frac{\Gamma(1+a-b)}{\Gamma(a)\Gamma(1-b)}\int_0^{1}t^{\frac{a}{2}-1}(1-t)^{-b}dt\\
=&\frac{1}{2}\frac{\Gamma(1+a-b)}{\Gamma(a)\Gamma(1-b)} \frac{\di\Gamma\left(\frac{a}{2}\right)\Gamma(1-b)}{\di\Gamma\left(1+\frac{a}{2}-b\right)}\\
=&\frac{1}{2}\frac{\di\Gamma\left(\frac{a}{2}\right)\Gamma(1+a-b)}{\Gamma(a)\di\Gamma\left(1+\frac{a}{2}-b\right)}.
\end{align*}
The results follows from analytic continuation.
\end{proof}

Now we prove a special case of our main result.
\begin{theorem}\label{Theorem3}
If none of $a/2$, $1+a-b$ and $1+a-c$ is zero or a negative integer, and $\text{Re}\,(2+a-2b-2c)>0$, then
\begin{align*}
_4F_3\left[\begin{aligned} a, 1+a/2, b, c\hspace{0.8cm}\\ a/2, 1+a-b, 1+a-c\end{aligned}\;;\; -1\right]=&\frac{\Gamma(1+a-b)\Gamma(1+a-c)}{\Gamma(1+a)\Gamma(1+a-b-c)}.
\end{align*}
\end{theorem}
\begin{proof}
This formula can be obtained as a limiting case of Dougall's formula (see for example, \cite{4}) for a finite $_7F_6$. Here we give an independent proof that do not use identities for $_{q+1}F_q$ with $q\geq 3$.
\begin{align*}
_4F_3\left[\begin{aligned} a, 1+a/2, b, c\hspace{0.8cm}\\ a/2, 1+a-b, 1+a-c\end{aligned}\;;\; -1\right]=&
\sum_{n=0}^{\infty} (-1)^n\frac{(a)_n(c)_n}{n!(1+a-c)_n}\frac{(b)_n(1+a/2)_n}{(1+a-b)_n(a/2)_n}\end{align*}
Using Corollary \ref{corollary2} with $c=1+a/2$, we have
\begin{align*}
&_4F_3\left[\begin{aligned} a, 1+a/2, b, c\hspace{0.8cm}\\ a/2, 1+a-b, 1+a-c\end{aligned}\;;\; -1\right]\\
=&\sum_{n=0}^{\infty} (-1)^n\frac{(a)_n(c)_n}{n!(1+a-c)_n}\sum_{r=0}^{\infty} \frac{(a+n)_r(a/2-b)_r(-n)_r}{r! (1+a-b)_r(a/2)_r}\\
=&\sum_{r=0}^{\infty}\frac{(a/2-b)_r}{r! (1+a-b)_r(a/2)_r}\sum_{n=r}^{\infty}(-1)^{n-r}\frac{(a)_{n+r}(c)_n}{(n-r)!(1+a-c)_n}\\
=&\sum_{r=0}^{\infty}\frac{(a/2-b)_r}{r! (1+a-b)_r(a/2)_r}\sum_{n=0}^{\infty}(-1)^{n}\frac{(a)_{n+2r}(c)_{n+r}}{n!(1+a-c)_{n+r}}\\
=&\sum_{r=0}^{\infty}\frac{(a/2-b)_r(a)_{2r}(c)_r}{r! (a/2)_r(1+a-b)_r(1+a-c)_r}\sum_{n=0}^{\infty}(-1)^{n}\frac{(a+2r)_{n}(c+r)_{n}}{n!(1+a-c+r)_{n}}\\
=&\sum_{r=0}^{\infty}\frac{(a/2-b)_r(a)_{2r}(c)_r}{r! (a/2)_r(1+a-b)_r(1+a-c)_r}\,_2F_1\left[\begin{aligned} a+2r, c+r\\1+a-c+r\end{aligned}\;;\;-1\right]\end{align*}
Using Lemma \ref{lemma9}, we find that
\begin{align*}
&_4F_3\left[\begin{aligned} a, 1+a/2, b, c\hspace{0.8cm}\\ a/2, 1+a-b, 1+a-c\end{aligned}\;;\; -1\right]\\=&\frac{1}{2}\sum_{r=0}^{\infty}\frac{(a/2-b)_r(a)_{2r}(c)_r}{r! (a/2)_r(1+a-b)_r(1+a-c)_r}\frac{\di\Gamma\left(\frac{a}{2}+r\right)\Gamma(1+a-c+r)}
{\di\Gamma(a+2r)\Gamma\left(1+\frac{a}{2}-c\right)}\\
=&\frac{1}{2}\frac{\di\Gamma\left(\frac{a}{2}\right)\Gamma(1+a-c)}
{\di\Gamma(a)\Gamma\left(1+\frac{a}{2}-c\right)}\sum_{r=0}^{\infty}\frac{(a/2-b)_r(c)_r}{r! (1+a-b)_r}\\
=&\frac{\di\Gamma\left(1+\frac{a}{2}\right)\Gamma(1+a-c)}
{\di\Gamma(1+a)\Gamma\left(1+\frac{a}{2}-c\right)}\,_2F_1\left[\begin{aligned} a/2-b, c \\1+a-b\end{aligned}\;;\;1\right]\end{align*}
Lemma \ref{lemma7} then gives
\begin{align*}
_4F_3\left[\begin{aligned} a, 1+a/2, b, c\hspace{0.8cm}\\ a/2, 1+a-b, 1+a-c\end{aligned}\;;\; -1\right]=&\frac{\di\Gamma\left(1+\frac{a}{2}\right)\Gamma(1+a-c)}
{\di\Gamma(1+a)\Gamma\left(1+\frac{a}{2}-c\right)}\frac{\di\Gamma(1+a-b)\Gamma\left(1+\frac{a}{2}-c\right)}{\di\Gamma\left(1+\frac{a}{2}\right)\Gamma(1+a-b-c)}\\
=&\frac{\Gamma(1+a-b)\Gamma(1+a-c)}{\Gamma(1+a)\Gamma(1+a-b-c)},
\end{align*}which proves the theorem.
\end{proof}
Finally, we prove the main result of this section.
\begin{theorem}\label{theorem7}Let $s\geq 0$. We have the formula
\begin{align*}
&_{2s+4}F_{2s+3}\left[\begin{aligned}a, 1+a/2, b_1, c_1, \ldots, b_{s+1}, c_{s+1} \hspace{2cm}\\a/2, 1+a-b_1, 1+a-c_1, \ldots, 1+a-b_{s+1}, 1+a-c_{s+1}\end{aligned}\,;\,-1\right]\\
=&\frac{\Gamma(1+a-b_{s+1})\Gamma(1+a-c_{s+1})}{\Gamma(1+a)\Gamma(1+a-b_{s+1}-c_{s+1})}\sum_{k_1, k_2, \ldots, k_s\geq 0} \\
&\times \prod_{j=1}^s \frac{(1+a-b_j-c_j)_{k_j}(b_{j+1})_{k_1+\ldots+k_j}(c_{j+1})_{k_1+\ldots+k_j}}{k_j!(1+a-b_j)_{k_1+\ldots+k_j}(1+a-c_j)_{k_1+\ldots+k_j}},
\end{align*}whenever the left hand side is convergent.
\end{theorem}
\begin{proof}In \cite{2}, this formula is obtained as the limiting case of another hypergeometric identity, which is the $q\rightarrow 1$ limit of a basic hypergeometric identity proved by Andrews \cite{12}. Here we give a direct proof using induction on $s$. The case $s=0$ was proved in Theorem \ref{Theorem3}. If $s\geq 1$, using Corollary \ref{corollary2}, we have
\begin{align*}
&_{2s+4}F_{2s+3}\left[\begin{aligned}a, 1+a/2, b_1, c_1, \ldots, b_{s+1}, c_{s+1}\hspace{2cm}\\a/2, 1+a-b_1, 1+a-c_1, \ldots, 1+a-b_{s+1}, 1+a-s_{s+1}\end{aligned}\,;\,-1\right]\\
=&\sum_{n=0}^{\infty}(-1)^n\frac{(a)_n(1+a/2)_n(b_1)_n(c_1)_n\ldots(b_{s+1})_n(c_{s+1})_n }{n! (a/2)_n(1+a-b_1)_n(1+a-c_1)_n \ldots (1+a-b_{s+1})_n (1+a-c_{s+1})_n  }\\
=&\sum_{n=0}^{\infty}(-1)^n\frac{(a)_n(1+a/2)_n(b_2)_n(c_2)_n\ldots(b_{s+1})_n(c_{s+1})_n }{n! (a/2)_n(1+a-b_2)_n(1+a-c_2)_n \ldots (1+a-b_{s+1})_n (1+a-c_{s+1})_n }\\&\times\sum_{r=0}^n\frac{(a+n)_r (1+a-b_1-c_1)_r(-n)_r}{r! (1+a-b_1)_r(1+a-c_1)_r}\\
=&\sum_{r=0}^{\infty}\sum_{n=r}^{\infty}\frac{(1+a-b_1-c_1)_r }{r! (1+a-b_1)_r(1+a-c_1)_r}\\
&\times(-1)^{n-r}\frac{(a)_{n+r}(1+a/2)_n(b_2)_n(c_2)_n\ldots(b_{s+1})_n(c_{s+1})_n }{(n-r)! (a/2)_n(1+a-b_2)_n(1+a-c_2)_n \ldots (1+a-b_{s+1})_n (1+a-c_{s+1})_n}\end{align*}\begin{align*}
=&\sum_{r=0}^{\infty} \sum_{n=0}^{\infty}\frac{(1+a-b_1-c_1)_r   }{r! (1+a-b_1)_r(1+a-c_1)_r}\\
&\times(-1)^n\frac{(a)_{n+2r}(1+a/2)_{n+r}(b_2)_{n+r}(c_2)_{n+r}\ldots(b_{s+1})_{n+r}(c_{s+1})_{n+r} }{n! (a/2)_{n+r}(1+a-b_2)_{n+r}(1+a-c_2)_{n+r} \ldots (1+a-b_{s+1})_{n+r} (1+a-c_{s+1})_{n+r}  }\\
=&\sum_{r=0}^{\infty} \frac{(1+a-b_1-c_1)_r  }{r! (1+a-b_1)_r(1+a-c_1)_r}\frac{(a)_{2r}(1+a/2)_{r} \di \prod_{j=2}^{s+1}(b_j)_{r}(c_j)_{r}}{ (a/2)_{r} \di\prod_{j=2}^{s+1}(1+a-b_j)_{r}(1+a-c_j)_{r}}\\
&\times \sum_{n=0}^{\infty}(-1)^n\frac{(a+2r)_{n}(1+a/2+r)_{n}  \di \prod_{j=2}^{s+1}(b_j+r)_{n}(c_j+r)_{n}}{n! (a/2+r)_{n}  \di\prod_{j=2}^{s+1}(1+a-b_j+r)_{n}(1+a-c_j+r)_{n}}.\end{align*}
By induction hypothesis,
\begin{align*}
&\sum_{n=0}^{\infty}(-1)^n\frac{(a+2r)_{n}(1+a/2+r)_{n}  \di \prod_{j=2}^{s+1}(b_j+r)_{n}(c_j+r)_{n}}{n! (a/2+r)_{n}  \di\prod_{j=2}^{s+1}(1+a-b_j+r)_{n}(1+a-c_j+r)_{n}}\\
=&_{2s+2}F_{2s+1}\left[\begin{aligned}a+2r, 1+a/2+r, b_2+r, c_2+r, \ldots, b_{s+1}+r, c_{s+1}+r \hspace{2cm}\\a/2+r, 1+a-b_2+r, 1+a-c_2+r, \ldots, 1+a-b_{s+1}+r, 1+a-c_{s+1}+r \end{aligned}\,;\,-1\right]\\
=&\frac{\Gamma(1+a-b_{s+1}+r)\Gamma(1+a-c_{s+1}+r)}{\Gamma(1+a+2r)\Gamma(1+a-b_{s+1}-c_{s+1})}\sum_{k_2, k_3, \ldots, k_s\geq 0} \\
&\times \prod_{j=2}^s \frac{(1+a-b_j-c_j)_{k_j}(b_{j+1}+r)_{k_2+\ldots+k_j}(c_{j+1}+r)_{k_2+\ldots+k_j}}{k_j!(1+a-b_j+r)_{k_2+\ldots+k_j}(1+a-c_j+r)_{k_2+\ldots+k_j}}.
\end{align*}
Hence,
\begin{align*}
&_{2s+4}F_{2s+3}\left[\begin{aligned}a, 1+a/2, b_1, c_1, \ldots, b_{s+1}, c_{s+1}\hspace{2cm}\\a/2, 1+a-b_1, 1+a-c_1, \ldots, 1+a-b_{s+1}, 1+a-c_{s+1}\end{aligned}\,;\,-1\right]\\
=&\sum_{r=0}^{\infty}\frac{(1+a-b_1-c_1)_r  }{r! (1+a-b_1)_r(1+a-c_1)_r}\frac{(a)_{2r}(1+a/2)_{r} \di \prod_{j=2}^{s+1}(b_j)_{r}(c_j)_{r}}{ (a/2)_{r} \di\prod_{j=2}^{s+1}(1+a-b_j)_{r}(1+a-c_j)_{r}}\\
&\times \frac{\Gamma(1+a-b_{s+1}+r)\Gamma(1+a-c_{s+1}+r)}{\Gamma(1+a+2r)\Gamma(1+a-b_{s+1}-c_{s+1})}\sum_{k_2, k_3, \ldots, k_s\geq 0}\\
&\times \prod_{j=2}^s \frac{(1+a-b_j-c_j)_{k_j}(b_{j+1}+r)_{k_2+\ldots+k_j}(c_{j+1}+r)_{k_2+\ldots+k_j}}{k_j!(1+a-b_j+r)_{k_2+\ldots+k_j}(1+a-c_j+r)_{k_2+\ldots+k_j}}\end{align*}\begin{align*}
=&\frac{\Gamma(1+a-b_{s+1})\Gamma(1+a-c_{s+1})}{\Gamma(1+a)\Gamma(1+a-b_{s+1}-c_{s+1})}\sum_{r, k_2, k_3, \ldots, k_s\geq 0}\frac{(1+a-b_1-c_1)_r  }{r! (1+a-b_1)_r(1+a-c_1)_r}\\&\times\frac{a}{a+2r}\frac{ (1+a/2)_{r}(b_2)_{r}(c_2)_{r}}{ (a/2)_{r} }\sum_{k_2, k_3, \ldots, k_s\geq 0}\\
&\times \prod_{j=2}^s \frac{(1+a-b_j-c_j)_{k_j}(b_{j+1})_{r+k_2+\ldots+k_j}(c_{j+1})_{r+k_2+\ldots+k_j}}{k_j!(1+a-b_j)_{r+k_2+\ldots+k_j}(1+a-c_j)_{r+k_2+\ldots+k_j}}
\end{align*}
But
\begin{align*}
\frac{a}{a+2r}\frac{ (1+a/2)_{r}}{ (a/2)_{r} }=\frac{a}{a+2r}\frac{a/2+r}{a/2}=1.
\end{align*} Setting $r=k_1$ complete the induction and prove the theorem.
\end{proof}

\bigskip
\section{Yet Another Alternative Proof of Zagier's Formula}
In this section, we give another simpler proof of Zagier's formula:
\begin{theorem}\label{Zagier2}
\begin{align}
H(a, b)=&2\sum_{r=1}^{K}(-1)^r\left\{\begin{pmatrix} 2r\\2a+2\end{pmatrix}\zeta(2r+1)+\begin{pmatrix} 2r\\2b+1\end{pmatrix}\zeta(\overline{2r+1})\right\}H(K-r),\label{eq62}\\
H^{\star}(a, b)=&-2\sum_{r=1}^{K} \left\{\left[\begin{pmatrix} 2r\\2a\end{pmatrix}-\delta_{r,a}\right]\zeta(2r+1)+\begin{pmatrix} 2r\\2b+1\end{pmatrix}\zeta(\overline{2r+1})\right\}H^{\star}(K-r),\label{eq63}
\end{align}where $K=a+b+1$.
\end{theorem}
Define the generating functions
\begin{align*}
G(x)=&\sum_{a=0}^{\infty}(-1)^a H(a)x^{2a},\\
G^{\star}(x)=&\sum_{a=0}^{\infty} H^{\star}(a)x^{2a},\\
F(x,y)=&\sum_{a=0}^{\infty}\sum_{b=0}^{\infty} (-1)^{a+b+1}H(a,b)x^{2a+2}y^{2b},\\
F^{\star}(x,y)=&\sum_{a=0}^{\infty}\sum_{b=0}^{\infty}  H^*(a,b)x^{2a}y^{2b+2}.
\end{align*}Here we use the convention that $H(0)=H^{\star}(0)=1$. Notice that our definition is slightly different from those defined in \cite{1} and \cite{3}.

It is well-known that
\begin{align}
G(x)=&\prod_{n=1}^{\infty}\left(1-\frac{x^2}{n^2}\right)=\frac{\sin\pi x}{\pi x},\label{eq51}\\
G^{\star}(x)=&\frac{1}{\di\prod_{n=1}^{\infty}\left(1-\frac{x^2}{n^2}\right)}=\frac{\pi x}{\sin \pi x}.\label{eq52}
\end{align}
Eq. \eqref{eq51} gives immediately
\begin{align*}
H(a)=&\frac{\pi^{2a}}{(2a+1)!}.
\end{align*}
Using the formula
\begin{align*}
\frac{\pi x}{\sin \pi x}=&1-2\sum_{j=1}^{\infty}(-1)^j\frac{x^2}{j^2-x^2}\\
=&1-2\sum_{a=1}^{\infty}\zeta(\overline{2a})x^{2a},
\end{align*}we obtain from Eq. \eqref{eq52} that for $a\geq 1$,
\begin{align}\label{eq56}
H^{\star}(a)=&-2\zeta(\overline{2a}).
\end{align}
As in \cite{1}, we find that for the generating functions $F(x, y)$ and $F^{\star}(x, y)$,
\begin{equation}\label{eq53}\begin{split}
F(x,y)
=&-x^2\sum_{m=1}^{\infty} \prod_{j=1}^{m-1}\left(1-\frac{x^2}{j^2}\right)\frac{1}{m^3}\prod_{k=m+1}^{\infty}\left(1-\frac{y^2}{k^2}\right)\\
=&\frac{\sin\pi y}{\pi y}\left.\frac{d}{dz}\right|_{z=0}\sum_{m=1}^{\infty} \frac{(-x)_m(x)_m}{(1-y)_m(1+y)_m}z(z+1)\ldots(z+m-1)\frac{1}{m!}\\
=&\frac{\sin\pi y}{\pi y}\left.\frac{d}{dz}\right|_{z=0}\,_3F_2\left[\begin{aligned} x,-x,z\quad\\1+y,1-y\end{aligned};1\right],
\end{split}\end{equation}
\begin{equation}\label{eq54}\begin{split}
F^{\star}(x,y)
=& y^2\sum_{m=1}^{\infty} \prod_{j=1}^{m}\left(1-\frac{x^2}{j^2}\right)^{-1}\frac{1}{m^3}\prod_{k=m}^{\infty}\left(1-\frac{y^2}{k^2}\right)^{-1}\\
=&-\frac{\pi y}{\sin\pi y} \left.\frac{d}{dz}\right|_{z=0}\sum_{m=1}^{\infty} \frac{(-y)_m(y)_m}{(1-x)_m(1+x)_m}z(z+1)\ldots(z+m-1)\frac{1}{m!}\\
=&-\frac{\pi y}{\sin\pi y} \left.\frac{d}{dz}\right|_{z=0}\,_3F_2\left[\begin{aligned} y,-y,z\quad\\1+x,1-x\end{aligned};1\right].
\end{split}\end{equation}
Comparing \eqref{eq53} and \eqref{eq54} give
\begin{lemma}\label{lemma10}
\begin{align*}
F(x,y)=&-\frac{\sin\pi y}{\pi y}\frac{\sin\pi x}{\pi x}F^{\star}(y, x).
\end{align*}
\end{lemma}

To prove Theorem \ref{Zagier2}, we first give a slightly simpler proof of the following theorem established in \cite{3} and \cite{16}.
\begin{theorem}\label{Pilehrood2}For any $a, b\geq 0$,
\begin{align}\label{eq66} H^{\star}(a, b)=-4\zeta\left(2a+1, \overline{2b+2}\right)-2\zeta\left(\overline{2a+2b+3}\right).\end{align}
\end{theorem}
\begin{proof}
By Theorem \ref{theorem7}, we have
\begin{align*}
&_8F_7\left[\begin{aligned} a, 1+a/2,  1+x, 1-x, -\alpha+x, -\alpha-x, y, -y\hspace{2cm}\\a/2,a-x, a+x,  1+a+\alpha-x, 1+a+\alpha+x, 1+a-y, 1+a+y \end{aligned}\;;\;-1\right]\\
=&\frac{\Gamma(1+a+y)\Gamma(1+a-y)}{\Gamma(1+a)^2}\\&\times\sum_{k_1, k_2\geq 0}\frac{(a-1)_{k_1}(-\alpha+x)_{k_1}(-\alpha-x)_{k_1}(1+a+2\alpha)_{k_2}(y)_{k_1+k_2}(-y)_{k_1+k_2}}
{k_1!k_2!(a-x)_{k_1}(a+x)_{k_1}(1+a+\alpha-x)_{k_1+k_2}(1+a+\alpha+x)_{k_1+k_2}}.
\end{align*}
Taking the $a\rightarrow 0$ limit, we find that
\begin{align*}
&1+2\sum_{k=1}^{\infty}\frac{(1+x)_k(1-x)_k(-\alpha+x)_{k}(-\alpha-x)_k(y)_k(-y)_k}{(-x)_k(x)_k(1+\alpha-x)_k(1+\alpha+x)_k(1-y)_k(1+y)_k}(-1)^k\\
=&\Gamma(1+y)\Gamma(1-y)\left\{\sum_{k=0 }^{\infty} \frac{(1+2\alpha)_{k} (y)_{k}(-y)_{k}}{k! (1+\alpha-x)_{k}(1+\alpha+x)_{k}}-\sum_{k=0}^{\infty}\frac{(1+2\alpha)_{k} (y)_{k+1}(-y)_{k+1}}{k!(1+\alpha-x)_{k+1}(1+\alpha+x)_{k+1} }\right\}\\
=&\frac{\pi y}{\sin \pi y}\left\{1+\sum_{k=1 }^{\infty} \frac{ (y)_{k}(-y)_{k}}{ (1+\alpha-x)_{k}(1+\alpha+x)_{k}}\left[\frac{(1+2\alpha)_{k}}{k!}-\frac{(1+2\alpha)_{k-1}}{(k-1)!}\right]\right\}.
\end{align*}
We then take the derivative of $\alpha$ at $\alpha=0$.
The left-hand side gives
\begin{align*}
2y^2\sum_{k=1}^{\infty} \frac{(-1)^k}{k^2-y^2}\left(4\sum_{j=1}^{k-1}\frac{j}{j^2-x^2}+\frac{2k}{k^2-x^2}\right),
\end{align*}and the right hand side gives
\begin{align*}
2\times\frac{\pi y}{\sin \pi y}\sum_{k=0 }^{\infty} \frac{ (y)_{k}(-y)_{k}}{ (1 -x)_{k}(1+x)_{k}}\frac{1}{k}=&2\times\frac{\pi y}{\sin\pi y} \left.\frac{d}{dz}\right|_{z=0}\,_3F_2\left[\begin{aligned} y,-y,z\quad\\1+x,1-x\end{aligned}\;;\;1\right]\\
=&-2F^{\star}(x,y).
\end{align*}
In other words,
\begin{align*}
F^{\star}(x,y)
=&-y^2\sum_{k=1}^{\infty} \frac{(-1)^k}{k^2-y^2}\left(4\sum_{j=1}^{k-1}\frac{j}{j^2-x^2}+\frac{2k}{k^2-x^2}\right)\\
=&-4 \sum_{a=0}^{\infty}\sum_{b=0}^{\infty}\sum_{k=1}^{\infty}\sum_{j=1}^{k-1}\frac{(-1)^k}{k^{2b+2}}\frac{1}{j^{2a+1}}x^{2a}y^{2b+2}
-2\sum_{a=0}^{\infty}\sum_{b=0}^{\infty}\sum_{k=1}^{\infty}\frac{(-1)^k}{k^{2a+2b+3}}x^{2a}y^{2b+2}\\
=&-\sum_{a=0}^{\infty}\sum_{b=0}^{\infty}\Bigl(4\zeta\left(2a+1,\overline{2b+2}\right)+2\zeta(\overline{2a+2b+3})\Bigr)x^{2a}y^{2b+2}.
\end{align*}
Comparing both sides give
\begin{align*}
H^{\star}(a, b)=&-4\zeta\left(2a+1,\overline{2b+2}\right)-2\zeta(\overline{2a+2b+3}),
\end{align*}which is the desired result.

\end{proof}

Now we can give a proof to the main theorem Theorem \ref{Zagier2}, which is different from that given in \cite{1, 11, 3}.
\begin{proof}
From Theorem \ref{Pilehrood2} and \eqref{eq35} in Theorem \ref{doubleEuler2}, we have
\begin{align*}
H^{\star}(a, b)
=&
-4 \zeta(2a+1)\zeta(\overline{2b+2})
\\&+4\sum_{l=0}^{K} \left[ \begin{pmatrix}2K-2l\\2a\end{pmatrix}\zeta(2K-2l+1) + \begin{pmatrix}2K-2l\\2b+1\end{pmatrix}\zeta(\overline{2K-2l+1})\right]\zeta(\overline{2l})
\\
=&4\sum_{r=0}^K\left[\begin{pmatrix}2r\\2a\end{pmatrix}\zeta(2r+1)
+\begin{pmatrix}2r\\2b+1\end{pmatrix}\zeta(\overline{2r+1})\right]\zeta(\overline{2K-2r})\\&-4 \zeta(2a+1)\zeta(\overline{2b+2})\\
=&-2\sum_{r=1}^{K} \left\{\left[\begin{pmatrix} 2r\\2a\end{pmatrix}-\delta_{r,a}\right]\zeta(2r+1)+\begin{pmatrix} 2r\\2b+1\end{pmatrix}\zeta(\overline{2r+1})\right\}H^{\star}(K-r).
\end{align*}In the last line, we have used the formula \eqref{eq56}.

From this formula for $H^{\star}(a, b)$ and \eqref{eq52}, we find that
\begin{align*}
F^{\star}(x,y)=&\frac{\pi y}{\sin \pi y}U(x,y)+\frac{\pi x}{\sin \pi x}V(x,y),
\end{align*}
where
\begin{align*}
U(x,y)=&-2\sum_{a=0}^{\infty}\sum_{\substack{r=a\\r\geq 1}}^{\infty} \left[\begin{pmatrix} 2r\\2a\end{pmatrix}-\delta_{r,a} \right] \zeta(2r+1)x^{2a}y^{2r-2a},\\
V(x,y)=&-2\sum_{b=0}^{\infty} \sum_{r=b+1}^{\infty}\begin{pmatrix} 2r\\2b+1\end{pmatrix}\zeta\left(\overline{2r+1}\right)x^{2r-2b-2}y^{2b+2}.
\end{align*}
By Lemma \ref{lemma10}, we have
\begin{align}\label{eq60}
F(x,y)=&-\frac{\sin\pi y}{\pi y}U(y,x)-\frac{\sin\pi x}{\pi x} V(y,x).
\end{align}
Now,
\begin{align*}
U(x,y)=&-2\sum_{r=1}^{\infty}\sum_{a=0}^r  \begin{pmatrix} 2r\\2a\end{pmatrix}x^{2a}y^{2r-2a} \zeta(2r+1)+2\sum_{r=1}^{\infty}\zeta(2r+1)x^{2r}\\
=&-\sum_{r=1}^{\infty}\left[(x+y)^{2r}+(x-y)^{2r}\right]\zeta(2r+1)+2\sum_{r=1}^{\infty}\zeta(2r+1)x^{2r},
\\
V(x,y)=&-2\sum_{r=1}^{\infty} \sum_{b=0}^{r-1}\begin{pmatrix} 2r\\2b+1\end{pmatrix}x^{2r-2b-2}y^{2b+2}\zeta\left(\overline{2r+1}\right)\\
=&-\frac{y}{x}\sum_{r=1}^{\infty}\left[(x+y)^{2r}-(x-y)^{2r}\right]\zeta\left(\overline{2r+1}\right).
\end{align*}
Hence,
\begin{align*}
U(y,x)=&-2\sum_{r=1}^{\infty}\sum_{a=0}^r \begin{pmatrix} 2r\\2a\end{pmatrix} \zeta(2r+1)x^{2a}y^{2r-2a}+2\sum_{r=1}^{\infty}\zeta(2r+1)y^{2r}\\
=&-2\sum_{a=1}^{\infty}\sum_{r=a}^{\infty} \begin{pmatrix} 2r\\2a\end{pmatrix} \zeta(2r+1)x^{2a}y^{2r-2a}\\
=&-2\sum_{a=0}^{\infty}\sum_{r=a+1}^{\infty} \begin{pmatrix} 2r\\2a+2\end{pmatrix} \zeta(2r+1)x^{2a+2}y^{2r-2a-2},\\
V(y,x)=&-\frac{x}{y}\sum_{r=1}^{\infty}\left[(x+y)^{2r}-(x-y)^{2r}\right]\zeta\left(\overline{2r+1}\right)\\
=&-2\sum_{b=0}^{\infty} \sum_{r=b+1}^{\infty}\begin{pmatrix} 2r\\2b+1\end{pmatrix}\zeta\left(\overline{2r+1}\right)x^{2r-2b}y^{2b}.
\end{align*}
Comparing both sides of \eqref{eq60} and using \eqref{eq51} give
\begin{align*}
H(a,b)=&2\sum_{r=1}^{K}(-1)^r\left\{\begin{pmatrix} 2r\\2a+2\end{pmatrix}\zeta(2r+1)+\begin{pmatrix} 2r\\2b+1\end{pmatrix}\zeta(\overline{2r+1})\right\}H(K-r).
\end{align*}This completes the proof.
\end{proof}

An advantage of the proof given here is that it better reflects the symmetries between the formulas for $H(a, b)$ and $H^{\star}(a, b)$.

In the rest of this section, we derive identities for the sums of $H(a, b)$ and $H^{\star}(a, b)$ with fixed $a+b$.

\begin{theorem}
\begin{align}
\sum_{a+b=K-1}H(a,b)=&\sum_{r=1}^K(-1)^{r-1}H(K-r)\zeta(2r+1)\label{eq61},\\
\sum_{a+b=K-1}H^{\star}(a,b)=& \sum_{r=1}^{K}H^{\star}(K-r)\zeta(2r+1).\label{eq62}
\end{align}
\end{theorem}
\begin{proof}
Using \eqref{eq53}, we have
\begin{equation}\label{eq70}\begin{split}
F(x,x)=&\sum_{a=0}^{\infty}\sum_{b=0}^{\infty} (-1)^{a+b+1}H(a,b)x^{2a+2b+2}\\
=&\frac{\sin\pi x}{\pi x}\left.\frac{d}{dz}\right|_{z=0} \,_3F_2\left[\begin{aligned} x,-x,z\quad\\1+x,1-x\end{aligned}\;;\;1\right].
\end{split}\end{equation}
By Theorem 3.4.1 in \cite{4}, we have
\begin{align*}
\,_3F_2\left[\begin{aligned} x,-x,z\hspace{1cm}\\1+z+x,1+z-x\end{aligned}\;;\;1\right]=&\frac{\di\Gamma\left(1+\frac{z}{2}\right)^2\Gamma(1+z+x)\Gamma(1+z-x)}{\di
\Gamma(1+z)^2\Gamma\left(1+\frac{z}{2}+x\right)\Gamma\left(1+\frac{z}{2}-x\right)}.
\end{align*}
Observe that
\begin{align*}
\left.\frac{d}{dz}\right|_{z=0} \,_3F_2\left[\begin{aligned} x,-x,z\quad\\1+x,1-x\end{aligned}\;;\;1\right]=\left.\frac{d}{dz}\right|_{z=0} \,_3F_2\left[\begin{aligned} x,-x,z\hspace{1cm}\\1+z+x,1+z-x\end{aligned}\;;\;1\right].
\end{align*}Using
$$\psi(1+x)=\frac{d}{dx}\log\Gamma(1+x)=\psi(1)+\sum_{n=1}^{\infty}\left(\frac{1}{n}-\frac{1}{n+x}\right),$$we find that
\begin{align*}
\left.\frac{d}{dz}\right|_{z=0} \,_3F_2\left[\begin{aligned} x,-x,z\quad\\1+x,1-x\end{aligned}\;;\;1\right]=&\frac{1}{2}\left(\psi(1+x)-\psi(1)\right)+\frac{1}{2}\left(\psi(1-x)-\psi(1)\right)\\
=&-\sum_{r=1}^{\infty}\zeta(2r+1)x^{2r}.
\end{align*}
It follows from \eqref{eq70} and \eqref{eq51} that
\begin{align*}
\sum_{a=0}^{\infty}\sum_{b=0}^{\infty}(-1)^{a+b+1}H(a,b)x^{2a+2b+2}=-\sum_{k=0}^{\infty}(-1)^kH(k)x^{2k} \sum_{r=1}^{\infty}\zeta(2r+1)x^{2r}.
\end{align*}
Comparing both sides give
\begin{align*}
\sum_{a+b=K-1}H(a,b)=&\sum_{r=1}^K(-1)^{r-1}H(K-r)\zeta(2r+1).
\end{align*}
In a similar way, we get
\begin{align*}
\sum_{a+b=K-1}H^{\star}(a,b)=& \sum_{r=1}^{K}H^{\star}(K-r)\zeta(2r+1).
\end{align*}
\end{proof}
\begin{corollary}\label{cor2}
\begin{align}\label{eq65}
\zeta(\overline{2K+1})=&-\frac{1}{2K} \sum_{a+b=K-1}\left(1+\frac{1}{2} \delta_{a,0}\right)H^{\star}(a,b).
\end{align}
\end{corollary}
\begin{proof}
Putting $a=0$ in \eqref{eq63}, we have
\begin{align*}
H^{\star}(0, K-1)=-2\sum_{r=1}^K\zeta(2r+1)H^{\star}(K-r)-4K\zeta(\overline{2K+1}).
\end{align*}
Eq. \eqref{eq62} then gives
\begin{align*}
-2K\zeta(\overline{2K+1})=&\sum_{a+b=K-1}H^{\star}(a,b)+\frac{1}{2}H^{\star}(0, K-1),
\end{align*}which gives the desired result.
\end{proof}
As mentioned in \cite{1}, \eqref{eq65} shows that the odd alternating zeta value $\zeta(\overline{2K+1})$ can be written as a $\mathbb{Q}$-linear combination of $H^{\star}(a, b)$. Eq. \eqref{eq56} expresses the even alternating zeta value $\zeta(\overline{2K})$ as
$$\zeta(\overline{2K})=-\frac{1}{2}H^{\star}(K).$$
From Theorem \ref{Pilehrood2} and Corollary \ref{cor2}, we have
\begin{corollary}
If $r\geq 0$, $s\geq 1$, then
\begin{align*}
\zeta(2r+1, \overline{2s})=\frac{1}{4K} \sum_{a+b=K-1}\left(1+\frac{1}{2} \delta_{a,0}-K\delta_{a, r}\right)H^{\star}(a,b),
\end{align*}where $r+s=K$.
\end{corollary}
\begin{proof}
From \eqref{eq66}, we have
\begin{align*}
\zeta(2r+1, \overline{2s})=&-\frac{1}{4}H^{\star}(r, s-1)-\frac{1}{2}\zeta(\overline{2K+1}).
\end{align*}One then obtains from \eqref{eq65} that
\begin{align*}
\zeta(2r+1, \overline{2s})=&-\frac{1}{4}H^{\star}(r, s-1)+\frac{1}{4K} \sum_{a+b=K-1}\left(1+\frac{1}{2} \delta_{a,0}\right)H^{\star}(a,b)\\
=&\frac{1}{4K} \sum_{a+b=K-1}\left(1+\frac{1}{2} \delta_{a,0}-K\delta_{a, r}\right)H^{\star}(a,b).
\end{align*}
\end{proof}
In this corollary,  the alternating double zeta value $\zeta(2r+1, \overline{2s})$ is explicitly expressed in terms of $H^{\star}(a, b)$. An interesting question is whether it is possible to explicitly express all alternating double zeta values in terms of $H^{\star}(a)$ and $H^{\star}(a, b)$.
\bibliographystyle{amsplain}

\end{document}